\documentclass[10pt]{article}
\usepackage{amssymb,amsthm}
\usepackage[tbtags]{amsmath}
\usepackage{color}
\usepackage[colorlinks,
            linkcolor=red,
            anchorcolor=blue,
            citecolor=green
            ]{hyperref}
\newtheorem{thm}{Theorem}[section]
\newtheorem{corollary}{Corollary}[section]
\newtheorem{lem}{Lemma}[section]

\theoremstyle{definition}

\theoremstyle{remark}
\newtheorem{rem}{Remark}[section]

\usepackage{color}

\definecolor{c20}{rgb}{0.,0.7,0.}
\definecolor{c30}{rgb}{0.,0.,1.}
\definecolor{c40}{rgb}{1,0.1,0.7}
\definecolor{c50}{rgb}{1,0,0}
\definecolor{c20}{rgb}{0.,0.7,0.}
\definecolor{c30}{rgb}{0.,0.,1.}
\definecolor{c40}{rgb}{1,0.1,0.7}
\definecolor{c50}{rgb}{1,0,0}
\def\IF{\infty}

\newcommand{\EE}[1]{\mathbb{E}\left\{#1\right\}}

\newcommand{\kb}[1]{\boldsymbol{#1}}
\newcommand{\vk}[1]{\kb{#1}}

\newcommand{\abs}[1]{\lvert #1 \rvert}

\newcommand{\R}{\mathbb{R}}

\newcommand{\BQN}{\begin{eqnarray}}
\newcommand{\EQN}{\end{eqnarray}}
\newcommand{\BQNY}{\begin{eqnarray*}}
\newcommand{\EQNY}{\end{eqnarray*}}

\newcommand{\BT}{\begin{theo}}
\newcommand{\ET}{\end{theo}}
\newcommand{\BK}{\begin{corollary}}
\newcommand{\EK}{\end{corollary}}

\newcommand{\BEL}{\begin{lem}}
\newcommand{\EEL}{\end{lem}}

\newcommand{\prooftheo}[1]{ \textsc{\bf Proof of Theorem} \ref{#1}:}
\newcommand{\prooflem}[1]{\textsc{\bf Proof of Lemma} \ref{#1}:}
\newcommand{\proofkorr}[1]{\textsc{\bf Proof of Corollary} \ref{#1}:}

\def\pE#1{#1}

\newtheorem{theo}{Theorem}[section]

\newcommand{\COM}[1]{}
\newcommand{\QED}{\hfill $\Box$}

\newcommand{\norm}[1]{\lVert #1 \rVert}

\def\wF{{\wwF}}
\def\rE{\rangle}
\def\rEE{\rangle}
\def\nO{\|}
\def\nOO{\|}
\def\kHA{\mathcal{H}_1}

\def\kHC{\mathcal{H}}
\def\ofp{(\Omega,\mathfrak{F},\mathbb{P})}
\newcommand{\pk}[1]{\mathbb{P} \left\{ #1 \right \} }
\def\wF{\underline{f}}

\begin{document}

\centerline{\bf \large
Asymptotic of Non-Crossings probability of Additive Wiener Fields}

 \bigskip
 \centerline{ Pingjin Deng\footnote{School of Finance, Nankai University, 300350, Tianjin, PR China,
and Department of Actuarial Science,
University of Lausanne
UNIL-Dorigny, 1015 Lausanne, Switzerland, email: Pingjin.Deng@unil.ch }
}

\bigskip
\centerline{\today{}}

{\bf Abstract}:
Let $W_i=\{W_i(t_i), t_i\in \R_+\}, i=1,2,\ldots,d$ are independent Wiener processes. $W=\{W(\mathbf{t}),t\in \R_+^d\}$ be the additive Wiener field define as the sum of $W_i$.  For any trend $f$ in $\kHC$ (the reproducing kernel Hilbert Space of $W$), we derive upper and lower bounds for the  boundary non-crossing probability $$P_f=P\{\sum_{i=1}^{d}W_i(t_i) +f(\mathbf{t})\leq u(\mathbf{t}), \mathbf{t}\in\R_+^d\},$$
where  $u: \R_+^d\rightarrow \R_+$ is a measurable function. Furthermore, for large trend functions $\gamma f>0$, we show that the asymptotically relation
$\ln P_{\gamma f}\sim \ln P_{\gamma \underline{f}}$ as $\gamma \to \IF$, where $\underline{f}$ is the projection of $f$ on some closed convex subset of $\kHC$.



\section{Introduction}\label{s:1}
For $d$ be a positive integer, let $X_i=\{X_i(t),\,t\in \R_+\},\,i=1,2,\ldots,d$ be independent real-valued stochastic processes on the same probability space $\ofp$. Define the \pE{$d$-parameters} real-valued additive field (additive process)
\BQNY
X(\mathbf{t})=X(t_1,t_2,\ldots,t_d)=\sum_{i=1}^{d}X_i(t_i),\quad \mathbf{t}=(t_1,t_2,\ldots,t_d)\in \R_+^d.
\EQNY
The additive process which plays a key role in studying of the general multiparameter processes, multiparameter potential theory, fractal geometry, spectral asymptotic theory has been actively investigated recently. To have a glance of these results, we refer the reader to \cite{khoshnevisan2003measuring, chen2003small,karol2008small,khoshnevisan1999brownian,khoshnevisan2002level,khoshnevisan2004additive,khoshnevisan2009harmonic} and the references therein.
\newline
On the other hand, calculation of boundary non-crossing probabilities of Gaussian processes is a key topic both of theoretical and  applied probability, see, e.g., 
\cite{Nov99,MR2009980,BNovikov2005,1103.60040,1079.62047,MR2576883,janssen}. Numerous applications concerned with the evaluation of boundary non-crossing probabilities relate  to mathematical finance, risk theory, queuing theory, statistics, physics among many other fields. In the literature, most of contributions are only concentrate on the boundary non-crossing probabilities of Gaussian processes with one-parameter (e.g. Brownian motion, Brownian bridge and fractional Brownian motion), some important results of this field can see in \cite{MR2028621,MR2175400,MR2016767,1137.60023,BiHa1,HMS12,MR3531495,MR3500417}. For multiparameter Gaussian processes, few cases are known about the boundary non-crossing probabilities (see, e.g., \cite{Pillow,hashorva2014boundary,Bi2}).
\newline
In this paper, we are concentrating on the calculation of boundary non-crossing probabilities of additive Wiener field $W$ which defined by \begin{equation}\label{eqhm-1}
W(\mathbf{t})= W_1(t_1)+W_2(t_2)+\ldots+W_d(t_d), \quad \mathbf{t} \in\R^d_+,
\end{equation}
where $W_i=\{W_i(t), t\in \R_+\}, i=1,2,\ldots,d$ are independent Wiener processes define on the same probability space $\ofp$. It can be checked easily that $W$ is a Gaussian field with the convariance function given by
\begin{equation}\label{eqhm-3}
\EE{W(\mathbf{s})W(\mathbf{t})}=\sum_{i=1}^{d}s_i\wedge t_i, \quad
\mathbf{s}=(s_1,s_2,\ldots,s_d), \;\mathbf{t}=(t_1,t_2,\ldots,t_d).
\end{equation}
For two measurable functions $f,u:\R_+^d\rightarrow \R$ we shall investigate the upper and lower bounds for
$$P_f=\pk{W(\mathbf{t})+f(\mathbf{t})\leq u(\mathbf{t}),\;\mathbf{t}\in\R_+^d}$$
In the following, we consider $u$ a general measurable function and $f\neq 0$ to belong to the reproducing kernel Hilbert space (RKHS) of $W$ which is denote by $\kHC$. A precise description of $\kHC$ is given in section \ref{sec:pre}, where the inner product $\langle f,g\rangle$ and the corresponding norm $\|f\|$ for $f,\,g \in \kHC$ are also defined.

As in \cite{HMS12}, a direct application of Theorem 1' in \cite{LiKuelbs}
shows that for any $f\in  \kHC $ we have
\newcommand{\Abs}[1]{\Bigl\lvert #1 \Bigr\rvert}
 \BQN\label{eq:00:2b}
\Abs{P_f - P_0} &\le \frac {1 }{\sqrt{2 \pi}} \norm{ f}.
\EQN
Further, for any $g\in\kHC$ such that $g\geq f$, we obtain
\BQN\label{eq:WL}
\Phi(\alpha - \norm{g}) \le P_{g}\le P_f \le \Phi(\alpha+ \norm{f}),
\EQN
where $\Phi$ is the distribution of an $N(0,1)$ random variable and $\alpha=\Phi^{-1}(P_0) $ is a finite constant. When $f\le 0$, then we can take always $g=0$ above which make the lower bound of \eqref{eq:WL} useful if $\norm{f}$ is large. When $\norm{f}$ is small, the equation \eqref{eq:00:2b} provides a good bound for the approximation rate of $P_f$ by $P_0$. Since the explicit formulas for computing $p_f$ seem to be impossible, the asymptotic performance of the bounds for trend functions $\gamma f$ with $\gamma\to \infty$ and $\gamma\to 0$ are thus worthy of consideration. This paper we shall consider the former case, and we obtain the following:
\newline
If  $f(\mathbf{t}_0)>0$ for some $\mathbf{t}_0$ with non-negative components, then
\pE{for any  $g\ge f, g\in \kHC$ we have}
\BQN\label{LD1}
\ln P_{\gamma f} \ge    \ln
\Phi(\alpha - \gamma \wF ) \ge  -(1+o(1))\frac{\gamma^2}{2}\norm{g}^2, \quad \gamma \to \IF,
\EQN
hence
\BQN\label{LD}
{\ln P_{\gamma f} \ge -(1+o(1))\frac{\gamma^2}{2}\norm{\wF}^2, \quad \gamma \to \IF},
\EQN
where $\wF$ (which is unique and exists) solves the following minimization problem
\BQN \label{OP}
\min_{ g,f \in { \kHC}, g \ge f}   \norm{g}= \norm{\wF}>0.
\EQN
\pE{In Section} 2 we shall show that  $\underline{f}$ is the projection of $f$ on a closed convex set of $\kHC$, and moreover
we show that
\BQN\label{LD2}
{\ln P_{\gamma f} \sim  \ln P_{\gamma \underline{f}} \sim - \frac{\gamma^2}{2}\norm{\underline{f}}^2, \quad \gamma \to \IF}.
\EQN
The rest of this paper are organized as follows: In section \ref{sec:pre} we  briefly talk about the RKHS of additive Wiener field and construct the solution of the minimization problem \eqref{OP}. We present our main results in Section \ref{sec:main}. The proofs of the results in this paper are shown in Section \ref{sec:proofs}, and we conclude this paper by Appendix.

\section{Preliminaries}\label{sec:pre}
This section reviews basic results of the reproducing kernel Hilbert space (RKHS),  and we shall give a representation of the RKHS of additive Wiener field $W$. We shall also construct $V$ as a closed convex set of $\kHC$, which finally enable us to prove that $\underline{f}$ in \eqref{OP} is the projection of $f$ on $V$. The idea of constructing $V$ comes from a similar result in one-parameter case (see e.g., \cite{BiHa1, janssen, Pillow, 1137.60023}).
\newline
In the following of \pE{this paper} bold letters are reserved for vectors, so we shall write for instance
$\mathbf{t}=(t_1, t_2, \ldots, t_d)\in\R^d_+$ and   $\lambda_1$  denote the Lebesgue measures on $\R_+$,
whereas $ds$ a mean integration with respect to this measure.
\subsection{The RKHS of additive Wiener field}
Recall that $W_1$ is an one-parameter Wiener process. It is well-known (see e.g., \cite{berlinet})
that the RKHS of the Wiener process $W_1$, denoted by $\kHA$, is characterized as follows
$$\kHA=\Bigl\{h:\R_+\rightarrow \R\big|h(t)=\int_{[0,t]}h'(s)ds,\quad  h'\in L_2(\R_+, \lambda_1) \Bigr\}, $$
with the inner product $\langle h,g\rangle_1=\int_{\R_+}h'(s)g'(s)ds$ and the corresponding norm $\|h\|_1^2=\langle h,h\rangle$. The description of RKHS for $W_i,\,i=2,3,\ldots,d$ are evidently the same.
We now begin to construct the RKHS of additive Wiener field $W$, for any
\BQNY
h_1(\mathbf{t})&=&f_1(t_1)+f_2(t_2)+\ldots+f_d(t_d),\\
h_2(\mathbf{t})&=&g_1(t_1)+g_2(t_2)+\ldots+g_d(t_d),
\EQNY
where $f_i(t_i),\,g_i(t_i)\in\kHA,\;i=1,2,\ldots,d,$ define the inner product
\begin{equation}\label{eqhm-10}
\langle h_1,h_2\rangle=\sum_{i=1}^{d}\int_{\R_+}f_i'(s)g_i'(s)ds.
\end{equation}
\begin{rem}
From lemma \ref{lemA1} in Appendix we have the representation $h(\mathbf{t})=h_1(t_1)+h_2(t_2)+\ldots+h_d(t_d)$ is unique, hence the above inner product is well defined.
\end{rem}
Next, in view of lemma \ref{lemA2} in Appendix we have the following
\begin{lem}
The RKHS for additive Wiener field $W$ is given by
\begin{equation}
\kHC=\Bigl\{h:\R_+^d\rightarrow \R\big|h(\mathbf{t})=\sum_{i=1}^{d}h_i(t_i),\; \text{where}\;  h_i\in \kHA,  i=1,2,\ldots,d \Bigr\}
\end{equation}
equipped with the norm $\norm{h}^2=\langle h,h\rangle.$
\end{lem}
For notational simplicity in the following  we shall use the same notation $\langle \cdot,\cdot\rangle$ and $\norm{\cdot}$ to present the inner product and norm respectively, on space $\kHA$ and $\kHC$.

\subsection{The solution of minimization problem}
In this subsection, we begin to solve equation \eqref{OP}. For any $h\in\kHA$, it has been shown (see \cite{1137.60023}), that
the smallest concave majorant of $h$ solves
\BQNY
\min_{ g,f \in { \kHA}, g \ge f}   \norm{g}= \norm{\wF}>0.
\EQNY
Moreover, as shown in \cite{janssen}
 the smallest concave majorant of $h$, which we denote by $\underline{h}$,
 can be written analytically as the unique projection of $h$ on the  closed convex set
$$V_1=\{h\in \kHA \big|\,h'(s) \; \text{is a non-increasing function} \}$$
i.e., $\underline{h}= Pr_{V_1}h$. Here we write $Pr_{A}h$ for the projection of $h$ on some closed set $A$ also for
other Hilbert spaces considered below. Further, if we define
$$\widetilde{V}_1=\{h\in \kHA  \big|\,\langle h,f\rE \leq 0 \;\text{for any}\; f\in V_1\} $$
be the polar cone of  $V_1$. Then the following hold
\begin{lem}\label{lemma 2.2}
\cite{hashorva2014boundary} With  the above notation and definitions we have
\begin{itemize}
\item[(i)]  If $h\in V_1$,  then $h\geq 0$.
\item[(ii)] If $h\in \widetilde{V}_1$,  then $h\leq 0$.
\item[(iii)] We have $\langle Pr_{V_1}h, Pr_{\widetilde{V}_1}h\rangle=0$ and further
\begin{equation}
h=Pr_{V_1}h+Pr_{\widetilde{V}_1}h.
\end{equation}
\item[(iv)] If $h=h_1+h_2$, $h_1\in V_1$, $h_2\in \widetilde{V}_1$ and $\langle h_1,  h_2\rangle=0$, then $h_1=Pr_{V_1}h$ and $h_2=Pr_{\widetilde{V}_1}h$.
\item[(v)] The unique solution of the minimization problem $\min_{g\geq h, g\in \kHA }\norm{g}$ is $\underline{h}=Pr_{V_1}h$.
\end{itemize}
\end{lem}

Since we are going to work with functions $f$ in $\kHC$ we need to consider the projection of such $f$ on a particular closed convex set.
In the following we shall write $f=f_1+f_2+\ldots+f_d$ meaning that $f(\vk{t})=f_1(t_1)+ f_2(t_2)+\ldots+f_d(t_d) $ where $f_1,f_2,\ldots,f_d \in \kHA$. Note in passing that this decomposition is unique for any $f\in \kHC$.
Define the closed convex set
$$V_{2}=\{h=h_1+h_2+\ldots+h_d\in \kHC \big|  h_1,h_2,\ldots,h_d \in V_1\}$$
and let $\widetilde{V_{2}}$ be the polar cone of $V_{2}$ given by
$$\widetilde{V_{2}}=\{h \in \kHC \big|\langle h, v\rEE  \leq 0 \;\text{for any}\; v\in V_{2}\},$$
with inner product from \eqref{eqhm-10}.
Analogous to Lemma \ref{lemma 2.2} we have
\begin{lem}\label{lemma 3.2}
For any  $h=h_1+h_2+\ldots+h_d\in \kHC$, we have
\begin{itemize}
\item[(i)]  If $h\in V_2$,  then $h_i\geq 0,i=1,2,\ldots,d$.
\item[(ii)] If $h\in \widetilde{V}_2$,  then $h_i\leq 0,i=1,2,\ldots,d$.
\item[(iii)] We have $\langle Pr_{V_2}h, Pr_{\widetilde{V}_2}h\rangle=0$ and further
\begin{equation}\label{eqhm-6}
h=Pr_{V_2}h+Pr_{\widetilde{V}_2}h.
\end{equation}
\item[(iv)] If $h=h_1+h_2$, $h_1\in V_2$, $h_2\in \widetilde{V}_2$ and $\langle h_1,  h_2\rangle=0$, then $h_1=Pr_{V_2}h$ and $h_2=Pr_{\widetilde{V}_2}h$.
\item[(v)] The unique solution of the minimization problem $\min_{g\geq h, g\in \kHC }\norm{g}$ is
\begin{equation}\label{eqhm-7}
\underline{h}=Pr_{V_2}h=Pr_{V_1}h_1+ Pr_{V_1}h_2+\ldots+Pr_{V_1}h_d.
\end{equation}
\end{itemize}
\end{lem}

\section{Main Result}\label{sec:main}
Consider two measurable d-parameter functions $f,u:\R_+^d\rightarrow \R$. Suppose that $f(\mathbf{0})=0$ and $f\in \kHC$.
Hence we can write
$$f(\mathbf{t})=\sum_{i=1}^{d}f_i(t_i),\quad f_i(t_i)\in \kHA,\,i=1,2,\ldots,d$$
we also suppose $f_i(0)=0,i=1,2,\ldots,d$ in the above decomposition. Recall their representations  $f_i(t_i)=\int_{[0,t_i]}f'_i(s)ds,\quad  f'_i\in L_2(\R_+, \lambda_1), i=1,2,\ldots,d.$ We shall estimate the boundary non-crossing probability
\BQNY
 P_f=\pk{W(\mathbf{t})+f(\mathbf{t})\leq u(\mathbf{t}),\;\mathbf{t}\in\R_+^d}.
 \EQNY
 In  the following we set  $\underline{f_i}= Pr_{V_1}f_i,i=1,2,\ldots,d$ and $\underline{f}= Pr_{V_2}f$.
We state next our main result:

\begin{thm} \label{Thn1} Let the following conditions hold:
\BQN\label{conA1}\lim_{t_i \rightarrow\infty}u(0,\ldots,t_i,0,\ldots,0)\underline{f_{i}'}(t_i)=0,\;i=1,2,\ldots,d.
\EQN
Then we have
\begin{equation*}\begin{gathered}
P_f\leq P_{ f-\underline{f} }
\exp\biggl (- \sum_{i=1}^{d}\int_{\R_+}u(0,\ldots,t_i,0,\ldots,0)d \underline{f_{i}'}(t_i)
-\frac12\|\underline{f}\nOO ^2\biggr).
\end{gathered}
\end{equation*}
\end{thm}

\begin{rem} Note that $f$ starts from zero therefore $f$ can not be a constant unless $f\equiv 0$ but this case is trivial.
\end{rem}
\begin{rem} Conditions \eqref{conA1} of the theorem means that asymptotically the components of shifts and their derivatives are negligible  in comparison with function $u$.
\end{rem}
Using Theorem \ref{Thn1}, we can obtain an asymptotically property of $P_{\gamma f}$, in fact, if $u(\mathbf{t})$ is bounded above, then we have the following result
\BK\label{korr}
If $f\in\kHC$ is such that $f(\mathbf{t}_0)$ for some $\mathbf{t}_0$, then
\BQNY
{\ln P_{\gamma f} \sim  \ln P_{\gamma \underline{f}} \sim - \frac{\gamma^2}{2}\norm{\underline{f}}^2, \quad \gamma \to \IF}.
\EQNY
\EK

\section{Proofs}\label{sec:proofs}
\prooflem{lemma 2.2} For $h\in V_1$, we have $h'$
 is non-increasing therefore  \pE{$h'$ } is  non-negative. Since $h(0)=0,$ thus $h(u)\geq 0$ for all $u$. The proof of statements (ii) to (v) can see in\cite{hashorva2014boundary}, we do not repeat the proof here.
\QED

\prooflem{lemma 3.2}
(i) If $h\in V_2$, from the definition of $V_2$, we obtain $h_1,h_2,\ldots,h_d \in V_1$. Thus $h_i\geq 0,i=1,2,\ldots,d$ follow directly from (i) in  Lemma \ref{lemma 2.2}
\newline
(ii) If $h(\mathbf{t})=h_1(t_1)+h_2(t_2)+\ldots+h_d(t_d)\in\widetilde{V}_2,$ then $h_i(t_i)\in\kHC.$  For any $f_i(t_i)\in V_1,$ let
\BQNY
v(\mathbf{t})=f_i(t_i)\in V_2.
\EQNY From the definition of $\widetilde{V}_2$, we obtain
$$\langle h,v\rangle=\langle h_i,f_i\rangle\leq 0.$$
Therefore, $h_i\in\widetilde{V}_1,$ and the results follow from (ii) in lemma \ref{lemma 2.2}.
\newline
The proof of statements $(iii)$ and $(iv)$ are similar to $(iii)$ and $(iv)$ in Lemma \ref{lemma 2.2}, and can obtain immediately from \cite{janssen}.
\newline
(v) For any $h(\mathbf{t})\in\kHC,$ let $g(\mathbf{t})\in\kHC$ such that $g\geq h$, we then have $g_i\geq h_i,i=1,2,\ldots,d,$ where
\BQNY
h&=&h_1+h_2+\ldots+h_d,\\
g&=&g_1+g_2+\ldots+g_d.
\EQNY
The minimization problem
\BQNY
\min_{g\geq h, g\in \kHC }\norm{g}&=&\min_{g\geq h, g\in \kHC}(\norm{g_1}+\norm{g_2}+\ldots+\norm{g_d})\\
&=&\sum_{i=1}^{d}\min_{g_i\geq h_i,g_i\in\kHA}\norm{g_i}\\
&=&\norm{\underline{h_1}}+\norm{\underline{h_2}}+\ldots+\norm{\underline{h_d}}.
\EQNY
The equalizes above hold if and only if
\begin{equation}
\underline{h}=Pr_{V_2}h=Pr_{V_1}h_1+ Pr_{V_1}h_2+\ldots+Pr_{V_1}h_d.
\end{equation}
Completing the proof.
\QED

\prooftheo{Thn1} Denote by $\widetilde{P}$ a probability measure that is defined via its Radon-Nikodym derivative
\begin{equation*}\begin{gathered} \frac{dP}{d\widetilde{P}}=\prod_{i=1}^{d}\exp\Big(-\frac12\|f_i\nO ^2+\int_{\R_+}f_i'(t_i)dW_i^0(t_i)\Big).
\end{gathered}
\end{equation*}
According to Cameron-Martin-Girsanov theorem,  $W_i^0(t)=W_i(t) +\int_{[0,t]}f_i'(s)ds,\;i=1,2,\ldots,d$ are \pE{independent} Wiener processes. Denote
$1_u\{X\}=1\{X(\mathbf{t})\leq u(\mathbf{t}),\;\mathbf{t}\in\R_+^d\}$ and $$W^0(\mathbf{t})=W_1^0(t_1)+ W_2^0(t_2)+\ldots+W_d^0(t_d).$$
Note that $ \norm{f}^2= \norm{f_1}^2+\norm{f_2}^2+\ldots+\norm{f_d}^2$, \pE{hence
using further} \eqref{eqhm-6} and \eqref{eqhm-7} we obtain
\BQNY
\lefteqn{P_f}\\
 &=&\EE{ 1_u\Big(\sum_{i=1}^{d}(W_i(t_i)+f_i(t_i))\Big)}\\
&=&\mathbb{E}_{\widetilde{P}}\Biggl( \frac{dP}{d\widetilde{P}}1_u\Big(W^0(\mathbf{t})\Big)\Biggr)\\
&=&
\exp\Big(-\frac12 \norm{f}^2 \Big) \EE{ \exp\Big(\sum_{i=1}^{d}\int_{\R_+}f_i'(t_i)dW_i^0(t_i)
\Big)1_u\Big(W^0(\mathbf{t})\Big)}\\
&=&\exp\Big(-\frac12 \norm{\underline{f}}^2 \Big)\\
&&\times \mathbb{E} \Biggl\{\prod_{i=1}^{d}\exp\Big(-\frac12\|Pr_{\widetilde{V}_1}{f_i} \nO ^2
+\int_{\R_+} Pr_{\widetilde{V}_1}f_i'(t_i)dW_i^0(t_i)\Big)
\times \exp\Big(\sum_{i=1}^{d}\int_{\R_+ } \underline{f_i}'(t_i)dW_i^0( {t_i})\Big)1_u\Big(W^0(\mathbf{t})\Big)\Biggr\}.
\EQNY
In order to re-write $\int_{\R_+ } \underline{f_1}'(t_1)dW_1^0( {t_1})$, we mention that in this integral $dW_1^0(t_1)=d_1(W^0(t_1,0,\ldots,0))$, therefore on the indicator $1_u\{\sum_{i=1}^{d} W_i^0(t_i)\}=1_u\{W^0(\mathbf{t})\}$
under conditions of the theorem and using lemma \ref{lemA3} in the Appendix we have the relations
\begin{equation}\begin{gathered}\label{eq3.2n}
\int_{\R_+}\underline{f_1}' (t_1)dW_1^0( {t_1})=\lim_{n\rightarrow \infty}\int_{[0,n]}\underline{f_1}' (t_1)dW_1^0( {t_1})\\
=\lim_{n\rightarrow \infty}\Big(\underline{f_1}' (n)W^0(n,0,\ldots,0)+\int_{[0,n]}W^0(t_1,0,\ldots,0)d(-\underline{f_1}' )(t_1)\Big).
\end{gathered}
\end{equation}
Similarly, for any $i=2,3,\ldots,d$ we have
\begin{equation}\label{eq3.3n}\int_{\R_+ }\underline{f_i}'(t_i)dW_i^0(t_i)
=\lim_{n\rightarrow \infty}\Big(\underline{f_i}' (n)W^0(0,\ldots,n,0,\ldots,0)+\int_{[0,n]}W^0(0,\ldots,t_i,0,\ldots,0)d(-\underline{f_i}' )(t_i)\Big).
\end{equation}

Combining \eqref{eq3.2n}--\eqref{eq3.3n}  and using conditions \eqref{conA1}, we get that  on the same indicator

\begin{equation}
\begin{gathered}\label{3.6n}\sum_{i=1}^{d}\int_{\R_+ } \underline{f_i}'(t_i)dW_i^0( {t_i})\leq \lim_{n\rightarrow \infty}\Big(\sum_{i=1}^{d}\underline{f_i}' (n)W^0(0,\ldots,n,0,\ldots,0)+\sum_{i=1}^{d}\int_{[0,n]}W^0(0,\ldots,t_i,0,\ldots,0)d(-\underline{f_i}' )(t_i)\Big)\\
\leq - \sum_{i=1}^{d}\int_{\R_+}u(0,\ldots,t_i,0,\ldots,0)d \underline{f_{i}'}(t_i).
\end{gathered}
\end{equation}
On the other hand, we have
\BQN\label{3.7n}
P_{f-\underline{f}}&=&\mathbb{E} \Biggl\{\prod_{i=1}^{d}\exp\Big(-\frac12\|f-\underline{f} \nO ^2
+\int_{\R_+} (f-\underline{f})'dW_i^0(t)\Big)
1_u\Big(W^0(\mathbf{t})\Big)\Biggr\}\\
\nonumber&=&\mathbb{E} \Biggl\{\prod_{i=1}^{d}\exp\Big(-\frac12\|Pr_{\widetilde{V}_1}{f_i} \nO ^2
+\int_{\R_+} Pr_{\widetilde{V}_1}f_i'(t)dW_i^0(t)\Big)
1_u\Big(W^0(\mathbf{t})\Big)\Biggr\}.
\EQN
From \eqref{3.6n} and \eqref{3.7n}, we conclude that
\begin{equation*}\begin{gathered}
P_f\leq P_{ f-\underline{f} }
\exp\biggl (- \sum_{i=1}^{d}\int_{\R_+}u(0,\ldots,t_i,0,\ldots,0)d \underline{f_{i}'}(t_i)
-\frac12\|\underline{f}\nOO ^2\biggr).
\end{gathered}
\end{equation*}
\QED

\proofkorr{korr} From \eqref{LD1} we obtain
\BQNY
\ln P_{\gamma f} \geq  -(1+o(1))\inf_{g\geq f}\frac{\gamma^2}{2}\norm{g}^2=-(1+o(1))\frac{\gamma^2}{2}\norm{\underline{f}}^2, \quad \gamma \to \IF.
\EQNY
On the other hand, from theorem \ref{Thn1} we obtain
\BQNY
P_{\gamma f} \leq  P_{\gamma(f-\underline{f})}\exp(-(1+o(1))\frac{\gamma^2}{2}\norm{\underline{f}}^2).
\EQNY
Since $f(\mathbf{t}_0)>0$, then $\lim_{\gamma\to\infty}P_{\gamma(f-\underline{f})}=constant>0$. Hence as $\gamma\to\infty,$
\BQNY
\ln P_{\gamma f} \leq -(1+o(1))\frac{\gamma^2}{2}\norm{\underline{f}}^2,
\EQNY
and the claim follows.
\QED

\section {Appendix}\label{sec:appendix}
\begin{lem}\label{lemA1} If the function $h:\R_+^d\rightarrow \R$ admits the representation
\begin{equation}\label{unique}
 h(\mathbf{t})=h_1(t_1)+h_2(t_2)+\ldots+h_d(t_d),
\end{equation}
 where $h_i\in \kHA, i=1,\ldots,d, $ then \pE{the} representation \eqref{unique} is unique.
\end{lem}
\begin{proof}
If the function $h:\R_+^d\rightarrow \R$ admits the representation
\begin{equation}\label{eqhm-2}
 h(\mathbf{t})=\sum_{i=1}^{d}f_i(t_i)=\sum_{i=1}^{d}g_i(t_i),
\end{equation}
 where $f_i,g_i\in \kHA, i=1,2,\ldots,d.$
For any $i=1,2,\ldots,d,$ we put $t_j=0$ for $j\neq i$, and note that $f_j(0)=g_j(0)=0,$ then we obtain $f_i=g_i,i=1,2,\ldots,d.$ Hence the representation \eqref{unique} is unique.
\end{proof}
Noting that the convariance function of $W_i$ is $s_i\wedge t_i$, and the convariance function of processes $W(\mathbf{t})=W_1(t_1)+W_2(t_2)+\ldots+W_d(t_d)$ is given by
\begin{equation*}
R(\mathbf{s},\mathbf{t}):=\EE{W(\mathbf{s})W(\mathbf{t})}=\sum_{i=1}^{d}s_i\wedge t_i, \quad
\mathbf{s}=(s_1,s_2,\ldots,s_d), \;\mathbf{t}=(t_1,t_2,\ldots,t_d).
\end{equation*}
Next, we will identify the RKHS corresponding to a sum of $d$ covariances.
Suppose now $R_i,i=1,2,\ldots,d$ are $d$ covariances of Gaussian processes, the corresponding RKHS are
$\mathbb{K}_i,i=1,2,\ldots,d.$ We suppose also $\norm{\cdot}_i$ the inner product of RKHS $\mathbb{K}_i,i=1,2,\ldots,d.$
The following is a well-known lemma and we refer the reader to \cite{aronszajn1950theory}
for its proof.
\begin{lem}\label{lemA2}
The RHKS of Gaussian processes which with covariances $R=R_1+R_2+\ldots+R_d$ is then given by the Hilbert space $\mathbb{K}$ consists of all functions $f(\mathbf(t))=f_1(t_1)+f_2(t_2)+\ldots+f_d(t_d),$ with $f_i(t_i)\in\mathbb{K}_i,i=1,2,\ldots,d,$ and the norm is given by
\BQNY
\norm{f}=\inf(\norm{f_1}_1+\norm{f_2}_2+\ldots+\norm{f_d}_d),
\EQNY
where the infimum taken for all the decomposition $f(\mathbf{t})=f_1(t_1)+f_2(t_2)+\ldots+f_d(t_d),\;g(\mathbf{t})=g_1(t_1)+g_2(t_2)+\ldots+g_d(t_d)$ with $f_i(t_i),\;g_i(t_i)\in\mathbb{K}_i,i=1,2,\ldots,d.$  Furthermore,  if for any $f\in \mathbb{K}$, the decomposition $f(\mathbf{t})=f_1(t_1)+f_2(t_2)+\ldots+f_d(t_d)$ is unique, then the  inner product of $\mathbb{K}$ is
\BQNY
\langle f,g \rangle=\langle f_1,g_1 \rangle+\langle f_2,g_2 \rangle+\ldots+\langle f_d,g_d \rangle.
\EQNY
\end{lem}
Also if we define the plus $\oplus$ among $\mathbb{K}_i,i=1,2,\ldots,d$ by $\mathbb{K}_i\oplus\mathbb{K}_j:=\{f=f_i+f_j\mid f_i\in\mathbb{K}_i,f_j\in\mathbb{K}_j\}$, then we can rewritten $\mathbb{K}$ as
\BQNY
\mathbb{K}=\mathbb{K}_1\oplus\mathbb{K}_2\oplus\ldots\oplus\mathbb{K}_d.
\EQNY
Let $W_1$ be a Wiener process, $h:\R_+\rightarrow \R$ be an integrable function, we can extend the integration of $h$ w.r.t $W_1$ on $\R_+$ by the following sence
\BQN\label{eq-integral}
\int_{\R_+}h(s)dW_1(s)=L_2-\lim_{n\rightarrow \infty}\int_{[0,n]}h (s)dW_1(s)
\EQN
whenever this limit exists. Furthermore, for any $h\in V_1,$ the derivative $h'\in  L_2(\R_+, \lambda_1)$ is non-increasing,  therefore $\int_{[0,n]}h'^2 (s)ds\leq h'^2 (0)n$ which implies that the integral $\int_{[0,n]}h (s)dW_1(s)$ is correctly defined as It\^o integral. We then can construct the  integration-by-parts formula
\begin{lem}\label{lemA3}
Let $h\in V_1,$ and  $W_1$ be a Wiener process. Then for any $T<\infty$, we have the following:
\BQN\label{eq-bypart}
\int_{[0,T]}h (s)dW_1(s)=\int_{[0,T]}W_1(s)d(-h (s))+h (T)W_1(T),
\EQN
where the integral in the right-hand side of \eqref{eq-bypart} is a Riemann-Stieltjes integral.
\end{lem}
\begin{proof}
From \cite{karatzas2012brownian}, for any partition $\pi$ of interval $[0,T],$ we obtain that the integral $\int_{[0,T]}h (s)dW_1(s)$ coincide with the limits in probability of integral sums
\BQNY
\int_{[0,T]}h (s)dW_1(s)&=&L_2-\lim_{\abs{\pi}\rightarrow 0}\sum_{i=1}^{N}h (s_{i-1})(W_1(s_i)-W_1(s_{i-1}))\\
&=&L_2-\lim_{\abs{\pi}\rightarrow 0}\sum_{i=1}^{N}W_1(s_i)(h (s_{i-1})-h (s_{i}))+W_1(T)h(T)\\
&=&\int_{[0,T]}W_1(s)d(-h (s))+W_1(T)h(T).
\EQNY
\end{proof}

{\bf Acknowledgment}:  This work was partly financed by the project NSFC No.71573143 and SNSF Grant 200021-166274.

\bibliographystyle{ieeetr}

 \bibliography{Additive2b}

\end{document}